\documentclass[leqno,12pt]{amsart}
\usepackage{graphicx,color}
\usepackage{latexsym}
\usepackage{graphicx} 
\usepackage{amsthm,amsmath, amssymb,amsfonts,esint}

\usepackage{amsmath,amssymb,amsfonts,amsthm}
\usepackage{layout,textcomp}

\theoremstyle{plain}

\theoremstyle{plain}
\newtheorem{theorem}{Theorem} [section]
\newtheorem{corollary}[theorem]{Corollary}

\newtheorem{proposition}[theorem]{Proposition}

\theoremstyle{definition}

\theoremstyle{remark}
\newtheorem{remark}[theorem]{Remark}

\numberwithin{theorem}{section}
\numberwithin{equation}{section}
\numberwithin{figure}{section}

\begin{document}

\title[Stable CMC surfaces with free boundary]{On stable constant mean curvature surfaces with free boundary}

\author[Ivaldo Nunes]{Ivaldo Nunes}
\address{Departamento de Matem\'atica, Universidade Federal do Maranh\~ao, Campus do Bacanga, S\~ao Lu\'is, MA, Brasil.}
\address{Departament of Mathematics, Imperial College London, South Kensington Campus, London, UK. (\textit{Current address})}

\email{ivaldo.nunes@ufma.br}

\date{}
\begin{abstract}
In \cite{RV}, Ros and Vergasta proved, among others interesting results, a theorem which states that an immersed orientable compact stable constant mean curvature surface $\Sigma$ with free boundary in a closed ball $B\subset\mathbb{R}^3$ must be a planar equator, a spherical cap or a surface of genus 1 with at most two boundary components. In this article, by using a modified Hersch type balancing argument, we complete their work by proving that $\Sigma$ cannot have genus 1. 
\end{abstract}
\maketitle


\section{Introduction}

A very interesting and important variational problem in differential  geometry is the free boundary problem for constant mean curvature (CMC) hypersurfaces. Given a compact Riemannian manifold $(M^{n+1},g)$ with nonempty boundary, the problem consists of finding critical points of the area functional among all compact hypersurfaces $\Sigma\subset M$ with $\partial\Sigma\subset\partial M$ which divides $M$ into two subsets of prescribed volumes.
Critical points for this problem are CMC hypersurfaces $\Sigma\subset M$ meeting $\partial M$ orthogonally along $\partial\Sigma$ and they are known as \textit{CMC hypersurfaces with free boundary}. This subject have been studied by many authors, for example, Nitsche \cite{Ni}, Struwe \cite{St}, Ros and Vergasta \cite{RV}, Souam \cite{So}, B\"urger and Kuwert \cite{BK} and Ros \cite{Ros}, among others.  We additionally  mention that, in the last few years,  there has been also great interest  on minimal hypersurfaces with free boundary (i.e. critical points of the area functional without any volume constraint). For instance, see  Fraser and Schoen \cite{FS1,FS2}, Brendle \cite{Br}, Fraser and Li \cite{FL}, Li \cite{L}, Chen, Fraser and Pang \cite{CFP}, Ambrozio \cite{Amb}, Folha, Pacard and Zolotareva \cite{FPZ}, M\'aximo, Nunes and Smith \cite{MNS}, Volkmann \cite{Vo},  Freidin, Gulian and McGrath \cite{FGMc} and McGrath \cite{Mc}.

When a CMC hypersurface $\Sigma\subset M$ with free boundary has nonnegative second variation of area for all preserving volume variations we call it \textit{stable}. In \cite{RV}, Ros and Vergasta studied immersed stable CMC hypersurfaces with free boundary in compact convex domains $M^{n+1}$ of $\mathbb{R}^{n+1}$.  A very natural question, which is still open, in the case that $M^{n+1}$ is a closed ball $B\subset\mathbb{R}^{n+1}$ is to ask whether the only immersed orientable compact stable CMC hypersurfaces with free boundary in $B\subset \mathbb{R}^{n+1}$ are the totally umbilical ones. For $n=2$, Ros and Vergasta \cite{RV} gave the following partial answer.

\begin{theorem}[Ros and Vergasta]\label{RV}
Let $B\subset\mathbb{R}^3$ be a closed ball. If $\Sigma\subset B$ is an immersed orientable compact stable CMC surface with free boundary, then $\partial\Sigma$ is embedded and the only possibilites are
\begin{itemize}
\item[(i)] $\Sigma$ is a totally geodesic disk;
\item[(ii)] $\Sigma$ is a spherical cap;
\item[(iii)] $\Sigma$ has genus 1 with at most two boundary components.
\end{itemize}
\end{theorem}

In order to prove Theorem \ref{RV}, Ros and Vergasta first used a Hersch type balancing argument to conclude that $g\in\{0,1\}$ and $r\in\{1,2\}$ where $g$ and $r$ denotes the genus of $\Sigma$ and the number of connected components of $\partial\Sigma$, respectively. Second, they prove that if $g=0$, then $r=1$. Therefore, in this case, by a theorem of Nitsche \cite{Ni}, $\Sigma$ has to be a totally umbilical disk.

Theorem \ref{RV} is a partial answer because the problem about the existence (or not) of stable CMC surfaces of genus 1 was not solved. Our main goal in this work is to improve the result above by proving that the possibility (iii) does not occur. More precisely, we prove:

\begin{theorem}\label{theoremA}
Let $B\subset\mathbb{R}^3$ be a closed ball. If $\Sigma\subset B$ is an immersed orientable compact stable CMC surface with free boundary, then $\Sigma$ has genus zero.
\end{theorem}

As a corollary of Theorems \ref{RV} and \ref{theoremA} we have the following complete classification of immersed compact stable CMC surfaces with free boundary in closed balls of $\mathbb{R}^3$.

\begin{corollary}\label{cor}
The totally umbilical disks are the only immersed orientable compact stable CMC surfaces in a closed ball $B\subset \mathbb{R}^3$
\end{corollary}

\begin{remark}
Corollary \ref{cor} can be regarded as the result  analogous to  Barbosa and do Carmo's theorem \cite{BdC} for immersed closed stable CMC surfaces in the Euclidean space. Since the latter holds for any dimension, we should expect the same for stable CMC surfaces with free boundary.
\end{remark}

In fact, Theorem \ref{theoremA} is a consequence of the following more general result.

\begin{theorem}\label{maintheorem}
Let $\Omega\subset\mathbb{R}^3$ be a smooth compact convex domain. Suppose  that the second fundamental form $\Pi^{\partial\Omega}$ of $\partial\Omega$ satisfies the pinching condition
\begin{equation}\label{pinching}
k\, h\leqslant \Pi^{\partial\Omega}\leqslant (3/2)\,k\, h
\end{equation}
for some constant $k>0$, where $h$ denotes the induced metric on $\partial\Omega$ 
. If $\Sigma\subset\Omega$ is an immersed orientable compact stable CMC surface with free boundary, then $\Sigma$ has genus zero and $\partial\Sigma$ has at most two connected components.
\end{theorem}

Let us give an ideia of the proof of Theorem \ref{maintheorem}. We first note (Proposition \ref{prop1}) that the stability of $\Sigma$ implies that the quadratic form given by the second variation of area is nonnegative for \textit{all} functions $\varphi$ such that $\varphi=0$ on $\partial\Sigma$ regardless of whether it satisfies $\int_\Sigma \varphi\,da=0$ or not. This fact allows us to apply a \textit{modified} Hersch type balancing argument which gives a better control on the genus of $\Sigma$. More precisely, instead of attaching a conformal disk at any connected component of $\partial\Sigma$, as in \cite{RV}, in order to find a conformal map $\psi=(\psi_1,\psi_2,\psi_3):\Sigma\to\mathbb{S}^2$ such that$\int_\Sigma\psi_i\,da=0$, for $i=1,2,3$, and having Dirichlet energy less than $8\pi(1+[(g+1)/2])$, where $g$ and $[x]$ stands for the genus of $\Sigma$ and the greatest integer less than or equal to $x$, respectively, we use as test functions the coordinates of a conformal map $\psi=(\psi_1,\psi_2,\psi_3):\Sigma\to\mathbb{S}^2_+$  satisfying $\int_\Sigma\psi_i\,da=0$ for $i=1,2,$ and having Dirichlet energy less than or equal to $4\pi(g+r)$, where $r$ denotes the number of connected components of $\partial\Sigma$. The key point here is that, since $\psi_3=0$ on $\partial\Sigma$, we are able to use $\psi_3$ as a test function because Proposition \ref{prop1}.

\begin{remark}
We note that the Hersch balance argument used by Ros and Vergasta proves that if $\Sigma$ is an immersed orientable compact stable CMC surface with free boundary in a compact convex domain $\Omega\subset\mathbb{R}^3$, then either $g\in\{1,2\}$ and $r=\{1,2,3\}$ or $g\in\{2,3\}$ and $r=1$ (cf. \cite{RV}, Theorem 5). In \cite{Ros}, Ros improved that result by using harmonic $1$-forms. More precisely, he proved that if $\Sigma$ is a stable CMC surface with free boundary in a \textit{mean} convex domain $\Omega\subset\mathbb{R}^3$, then either $g=0$ and $r\in\{1,2,3,4\}$ or $g=1$ and $r\in\{1,2\}$ (cf. Theorem 9 of \cite{Ros}). Note that if $\Omega$ is convex, then the case $g=0$ and $r=4$ does not occur. It is worth noting that our methods (see inequality \eqref{mainineq}) also imply that a stable CMC surface $\Sigma$ with free boundary in a convex domain $\Omega\subset\mathbb{R}^3$ cannot have genus $g=2$ or $g=3$.
\end{remark}

\begin{remark}
In \cite{So}, Souam dealt with immersed compact stable CMC surfaces with free boundary in compact domains of $\mathbb{S}^{n+1}$  and $\mathbb{H}^{n+1}$. Because of Barbosa, do Carmo and Eschenburg's theorem  \cite{BdCE} for closed immersed stable CMC surfaces in $\mathbb{S}^{n+1}$ and $\mathbb{H}^{n+1}$ it is natural to ask if the fact in Corollary \ref{cor} also holds for stable CMC surfaces $\Sigma$ with free boundary in balls $B$ of $\mathbb{S}^{n+1}$ or $\mathbb{H}^{n+1}$. As in \cite{RV}, Souam \cite{So} (cf. Theorem 5.1) proved that this is true in the case $\Sigma$ is has genus zero and $n=2$. It easy to see that our methods can be adapted in order to prove that an immersed orientable compact CMC surface $\Sigma\subset B$ with free boundary has genus zero in the case that either $B$ is  a ball of radius $r<\pi/2$ in $\mathbb{S}^{3}$ or $B$ is a ball of $\mathbb{H}^3$ and $|H|\geqslant 1$, where $H$ is the mean curvature of $\Sigma$. We also note that Souam proved (cf. \cite{So}, Theorem 3.1) that, for any dimension, a stable CMC hypersurface with free boundary in the hemisphere $\mathbb{S}^{n+1}_+$ has to be totally umbilical.
\end{remark}


\subsection*{Acknowledgements} The author would like to thank Andr\'e Neves and Fernando Cod\'a Marques for many useful mathematical conversations. The author was financially supported by CNPq-Brazil.

\section{Proof of Theorem \ref{maintheorem}}

Let $\Omega\subset\mathbb{R}^3$ be a smooth compact domain. Let  $\eta$ be the outward normal unit vector field along $\partial\Omega$ and let $\Pi^{\partial\Omega}$ be the second fundamental form of $\partial\Omega$ with respect to $\eta$, that is, $\Pi^{\partial\Omega}(X,Y)=\langle D_X\eta,Y\rangle$ for all tangent vectors $X,Y\in T(\partial\Omega)$, where $D$ denotes the Levi-Civita connection of the Euclidean space $\mathbb{R}^3$. We say that $\Omega$ is \textit{convex} if $\Pi^{\partial\Omega}(X,X)>0$ for all tangent vector $X\in T(\partial\Omega)$, $X\neq 0$.

Let $\Sigma\subset\Omega$ be an immersed orientable compact surface with nonempty boundary such that $\partial\Sigma=\Sigma\cap\partial\Omega$. We say that $\Sigma$ is a \textit{constant mean curvature (CMC) surface with free boundary} if $\Sigma$ is a critical point of the area functional among all surfaces $\widehat{\Sigma}\subset\Omega$ such that $\partial\widehat{\Sigma}=\widehat{\Sigma}\cap\partial\Omega$ which divides $\Omega$ into two subsets with prescribed volumes. We say that $\Sigma$ is \textit{stable} if it has nonnegative second variation of area for all preserving volume variations. By using the formulae of first and second variation of area, this means that
\begin{itemize}
\item[(i)] The mean curvature $H$ of $\Sigma$ is constant;
\item[(ii)] $\Sigma$ meets $\partial\Omega$ orthogonally along $\partial\Sigma;$
\item[(iii)] (Stability) The inequality
$$
Q(\varphi,\varphi)=\int_\Sigma|\nabla\varphi|^2-|A|^2\varphi^2\,da-\int_{\partial\Sigma}\Pi^{\partial\Omega}(N,N)\varphi^2\,dl\geqslant 0,
$$
holds for all functions $\varphi\in H^1(\Sigma)$ ,  such that $\int_\Sigma \varphi \,da=0$, where $H^1(\Sigma)$, $N$ and $A$ denote the first Sobolev space of $\Sigma$, the normal unit vector field along $\Sigma$ and the second fundamental form of $\Sigma$ with respect to $N$, respectively.
\end{itemize}

If we denote by $\nu$ the outward unit conormal of $\partial\Sigma$ tangent to $\Sigma$, then the condition (ii) above means that $\nu(p)=\eta(p)$ for all $p\in\partial\Sigma$.

In order to prove Theorem \ref{maintheorem} we need the following very useful proposition.

\begin{proposition}\label{prop1}
Let $\Omega\subset\mathbb{R}^3$ be a compact convex domain. If $\Sigma\subset\Omega$ is an immersed stable CMC surface with free boundary, then
\begin{equation}\label{ineq1}
Q^0(\varphi,\varphi)=\int_\Sigma|\nabla\varphi|
^2-|A|^2\varphi^2\,da\geqslant 0
\end{equation}
for all $\varphi$ such that $\varphi=0$ on $\partial\Sigma$.
\end{proposition}

\begin{proof}
Consider an orthormal basis $\{e_1,e_2,e_3\}$ of $\mathbb{R}^3$ and define $u_i(x)=\langle e_i,N(x)\rangle$, $x\in\Sigma$, for each $i=1,2,3$.

A direct computation shows that
$$
L u_i=0
$$
for all $i=1,2,3$, where $L=\Delta_\Sigma+|A|^2$ is the Jacobi operator of $\Sigma$.

Therefore, we have 
\begin{align*}
Q(u_i,u_i)&=\int_{\partial\Sigma}u_i\frac{\partial u_i}{\partial \nu}\,dl-\int_{\partial\Sigma}\Pi^{\partial\Omega}(N,N) u_i^2\,dl\\
&=\frac{1}{2}\int_{\partial\Sigma}\frac{\partial}{\partial\nu}(u_i^2)\,dl-\int_{\partial\Sigma}\Pi^{\partial\Omega}(N,N)u_i^2\,dl
\end{align*}
for all $i=1,2,3$.

By summing over $i=1,2,3$, we get that
\begin{align*}
\sum_{i=1}^3 Q(u_i,u_i)&=\frac{1}{2}\int_{\partial\Sigma}\frac{\partial}{\partial\nu}\left(\sum_{i=1}^3u_i^2\right)\,dl-\int_{\partial\Sigma}\Pi^{\partial\Omega}(N,N)\left(\sum_{i=1}^3u_i^2\right)\,dl\\
&=-\int_{\partial\Sigma}\Pi^{\partial\Omega}(N,N)\,dl<0
\end{align*}
since $\sum_{i=1}^3u_i^2=|N|^2=1$.

Thus, there exists $i\in\{1,2,3\}$ such that $Q(u_i,u_i)<0$. In particular, we have $\int_\Sigma u_i\,da\neq 0$ because the stability of $\Sigma$.

Now, let $v$ be the first eigenfunction of $L$ with Dirichlet boundary condition, that is, $v=0$ on $\partial\Sigma$. After multiplying $v$ by a constant we can assume that $\int_\Sigma v\,da=\int_{\Sigma}u_i\,da$. Note that, since $v=0$ on $\partial\Sigma$ and $Q(u_i,u_i)<0$, we have that $v-u_i\neq 0$.

Thus, since $\Sigma$ is a stable CMC surface with free boundary, we obtain that
\begin{align*}
0&\leqslant Q(v-u_i,v-u_i)\\
&=Q^0(v,v)-2Q^0(v,u_i)+Q(u_i,u_i)\\
&<Q^0(v,v),
\end{align*}
where we have used above that $Q(u_i,u_i)<0$ and $Q^0(v,u_i)=0$, because $Lu_i=0$ and $v=0$ on $\partial\Sigma$. This finishes the proof.
\end{proof}

From now on we consider a compact convex domain $\Omega\subset\mathbb{R}^3$ satisfying the pinching condition \eqref{pinching} and an immersed orientable compact stable CMC surface $\Sigma\subset\Omega$ with free boundary. We denote by $g$ the genus of $\Sigma$ and by $r$ the number of connected components of $\partial\Sigma$.



\begin{proof}[Proof of Theorem \ref{maintheorem}]
By a result of Gabard \cite{Ga}, which improved a previous result due to Alfhors \cite{Ahl}, there exists a proper conformal branched cover $\psi:\Sigma\to \mathbb{D}^2$, where $\mathbb{D}^2\subset\mathbb{R}^2$ is the closed unit disk, of degree at most $g+r$. Since $\mathbb{D}^2$ is conformally equivalent to the hemisphere $\mathbb{S}^2_+\subset\mathbb{R}^3$ of radius one, we can assume that $\psi=(\psi_1,\psi_2,\psi_3):\Sigma\to \mathbb{S}^2_+$. Moreover, by using conformal diffeomorphisms of $\mathbb{S}^2_+$, we may suppose that
$$
\int_\Sigma\psi_i\,da=0
$$
for $i=1,2$.

Thus, we have 
\begin{equation}\label{ineq2}
\int_{\partial\Sigma}\Pi^{\partial\Omega}(N,N)\,dl+\sum_{i=1}^2\int_\Sigma|A|^2\psi_i^2\,da\leqslant \sum_{i=1}^2\int_\Sigma|\nabla\psi_i|^2\,da.
\end{equation}

Since $\psi_3=0$ on $\partial\Sigma$, we get by Proposition \ref{prop1} that
\begin{equation}\label{ineq3}
\int_\Sigma |A|^2\psi_3^2\,da\leqslant\int_\Sigma|\nabla\psi_3|^2\,da.
\end{equation}

By adding \eqref{ineq2} and \eqref{ineq3} and using that $|A|^2=H^2-2K$ we obtain 
\begin{align*}
\int_{\partial\Sigma}\Pi^{\partial\Omega}(N,N)\,dl+\int_\Sigma H^2\,&da\leqslant \int_\Sigma|\nabla\psi|^2\,da+2\int_\Sigma K\,da\\
&\leqslant 4\pi(g+r)+4\pi(2-2g-r)-2\int_{\partial\Sigma}k_g\,dl.
\end{align*}

Therefore
\begin{equation}\label{mainineq}
\int_{\partial\Sigma}\Pi^{\partial\Omega}(N,N)\,dl+2\int_{\partial\Sigma}k_g\,dl+\int_\Sigma H^2\,da\leqslant 4\pi(2-g).
\end{equation}



Since $\Pi^{\partial\Omega}$ satisfies the pinching condition \eqref{pinching}, the inequality \eqref{mainineq} becomes
$$
4\pi(2-g)\geqslant\int_\Sigma H^2\,da+3k\,L(\partial\Sigma),
$$
because $\Pi^{\partial\Omega}(N,N),k_g=\Pi^{\partial\Omega}(T,T)\geqslant k$, where $T$ denotes the unit vector tangent to $\partial\Sigma$. We note that the equality $k_g=\Pi^{\partial\Omega}(T,T)$ comes from the free boundary condition.

Thus we have that
\begin{equation}\label{ineq3}
4\pi(2-g)\geqslant\frac{1}{2}\int_\Sigma H^2\,da+2\left(\frac{1}{4}\int_\Sigma H^2\,da+\frac{3}{2}k\,L(\partial\Sigma)\right).
\end{equation}

Now, since $\Pi^{\partial\Omega}\leqslant (3/2)k\,h$, we have by Corollary 5.8 of \cite{Vo} that
\begin{equation}\label{volkmannineq}
\frac{1}{4}\int_\Sigma H^2\,da+\frac{3}{2}k\,L(\partial\Sigma)\geqslant 2\pi,
\end{equation}
with equality if and only if $\Sigma$ is a spherical cap or a flat unit disk in which case we are done. Thus, we may suppose strict inequality in \eqref{volkmannineq}. This together with inequality \eqref{ineq3} implies that $4\pi(2-g)>4\pi$. Therefore $g<1$, that is, $\Sigma$ has genus zero.


Finally, by using the same balancing argument as Theorem 5 of \cite{RV} together with the strict inequality in \eqref{volkmannineq} again, we conclude that $16\pi -4\pi r>4\pi$, that is, $r<3$.

Thus, we conclude that $\Sigma$ has genus zero and $\partial\Sigma$ has at most two connected components.
\end{proof}



\end{document}